\documentclass[11pt, a4paper]{amsart}

\usepackage{amssymb,amsmath,amsthm}
\usepackage{hyperref}
\usepackage{graphicx}
\usepackage{xcolor}
\usepackage{graphicx} 
\usepackage{tikz}
\usetikzlibrary{arrows,calc,positioning}
\usetikzlibrary{external}

\newcommand{\va}{\mathbf{a}}

\definecolor{author}{rgb}{0.5,0.5,0.0}
\definecolor{def}{rgb}{0.0,0.5,0.5}
\definecolor{high}{rgb}{0.1,0.2,0.5}
\definecolor{darkgreen}{rgb}{0.0,0.4,0.0}
\definecolor{darkred}{rgb}{0.4,0.0,0.0}

\newcommand{\Kon}{\mathcal{K}_o^n}

\newcommand{\bd}{\mathrm{bd}}

\newcommand{\inte}{\mathrm{int}\,}
\newcommand{\conv}{\mathrm{conv}}

\newcommand{\lin}{\mathrm{lin}\,}
\newcommand{\R}{\mathbb{R}}

\newcommand{\N}{\mathbb{N}}

\newcommand{\dbm}{\mathrm{d}}

\newcommand{\V}{\mathrm{V}}
\newcommand{\U}{\mathrm{U}}
\newcommand{\ip}[2]{\left\langle #1,#2\right\rangle}
\newcommand{\dive}{\mathrm{div}}

\newtheorem{theorem}{Theorem}[section]

\newtheorem{lemma}[theorem]{Lemma}

\newtheorem{proposition}[theorem]{Proposition}  

\numberwithin{equation}{section}

\usepackage{pdfcomment}

\begin{document}

\title{Cone-volume measures of polytopes} 
\author{Martin Henk}
\address{Fakult\"at f\"ur Mathematik, Otto-von-Guericke Universit\"at
  Magdeburg, Universit\"atsplatz 2, D-39106 Magdeburg, Germany}
\email{martin.henk@ovgu.de, eva.linke@ovgu.de}
\author{Eva Linke}
\thanks{Second author is supported by Deutsche Forschungsgemeinschaft;
  He 2272/5-1.}
\dedicatory{In memory of Fiona Prohaska}
\begin{abstract}
The cone-volume measure of a polytope with
centroid at the origin is proved to satisfy the subspace concentration condition.
As a consequence a conjectured (a dozen years ago) fundamental sharp affine
isoperimetric inequality for the U-functional is completely
established --- along with its equality conditions.
\end{abstract}
\maketitle

\section{Introduction}
Let $\Kon$ be the set of all convex bodies in $\R^n$ having the origin
in their interiors, i.e., $K\in \Kon$ is a convex compact subset of the
$n$-dimensional Euclidean space $\R^n$ with $0\in\inte(K)$. For $K\in\Kon$ the {\em cone-volume
measure}, $\V_K$, of $K$ is a Borel measure  on the unit sphere
$S^{n-1}$ defined for a Borel set $\omega\subseteq S^{n-1}$ by
\begin{equation}
\V_K(\omega) =\frac{1}{n}\int_{x\in \nu_K^{-1}(\omega)}
\ip{x}{\nu_K(x)}\,\mathrm{d}\mathcal{H}^{n-1}(x),
\label{eq:cvm}
\end{equation}
where $\nu_K:\bd^\prime K\to S^{n-1}$ is the Gauss map of $K$, defined
on $\bd^\prime K$, the set of points of the boundary of $K$ 
having a unique outer normal, $\ip{x}{\nu_K(x)}$ is the standard inner
product on $\R^n$, and  $\mathcal{H}^{n-1}$ is the $(n-1)$-dimensional
Hausdorff measure. In recent years, cone-volume measures have 
appeared and were studied in various contexts, see, e.g.,
\cite{Barthe:2005di, Boroczky:2012vm, Boroczky:2012us, Gromov:1987vv,
  Ludwig:2010dk, Ludwig:2010tm, Naor:2007ey, Naor:2003gb,
  Paouris:2012ia, Stancu:2012ur}.  

In particular, in the very recent and groundbreaking paper \cite{Boroczky:2012us}  on the logarithmic Minkowski problem, 
B\"or\"oczky Jr., Lutwak, Yang \& Zhang characterize the cone-volume
measures of origin-symmetric convex bodies as exactly those
non-zero finite even Borel measures on $S^{n-1}$ which satisfy the
{\em subspace concentration condition}. Here a finite Borel measure $\mu$ on
$S^{n-1}$ is said to satisfy the {\em subspace concentration condition}
if for every subspace $L\subseteq\R^n$ 
\begin{equation}
  \mu(L\cap S^{n-1})\leq\frac{\dim L}{n}\mu(S^{n-1}),
\label{eq:scc}
\end{equation} 
and  equality holds in \eqref{eq:scc} for a subspace $L$ if and only
if there exists a subspace $\overline{L}$, complementary to $L$, so
that also 
\begin{equation*}
  \mu(\overline{L}\cap S^{n-1})=\frac{\dim \overline{L}}{n}\mu(S^{n-1}),
\end{equation*} 
i.e., $\mu$ is concentrated on $S^{n-1}\cap (L\cup\overline{L})$.

This concentration condition is at the core of different problems
in Convex Geometry;  it provides not only the solution to the
logarithmic Minkowski problem for origin-symmetric convex bodies \cite{Boroczky:2012us}, but, for instance, in \cite[Theorem
1.2]{Boroczky:2012vm}, it was shown that the subspace concentration
condition is also equivalent to the property that a finite Borel measure has an
affine isotropic image. 

Now let $P\in \Kon$ be a polytope with facets $F_1,\dots,F_m$, and let
$a_i\in S^{n-1}$ be the outer unit normal of the facet $F_i$,
$1\leq i\leq m$.  For each facet we consider $C_i=\conv\{0,F_i\}$,
i.e.,  the convex hull of $F_i$ with the
origin, or in other words, $C_i$ is the cone/pyramid with basis $F_i$ and
apex  $0$.   

\begin{figure}[hbt]
\begin{tikzpicture}[scale=0.8, line join=bevel]
\coordinate (V1) at (-1,-1);
\coordinate (V2) at (-1,1);
\coordinate (V3) at (1.5,2);
\coordinate (V4) at (3,0.5);
\coordinate (V5) at (1,-1.5);

\coordinate (null) at (0,0);

\coordinate (a1) at (-1,0); 
\coordinate (a2) at (-0.274,0.9615);
\coordinate (a3) at (0.707,0.707);
\coordinate (a4) at (0.707,-0.707);
\coordinate (a5) at (-0.2425,-0.97014);

\draw [fill opacity=0.1,fill=gray] (V1) -- (V2) -- (V3) -- (V4)--(V5)--cycle;
\draw [color=blue, ultra thick] (V1) -- (V2) -- (V3) -- (V4) -- (V5)--cycle;

 \draw [color=blue, thick, ->] ($0.5*(V1)+0.5*(V2)$) --
 ($0.5*(V1)+0.5*(V2)+(a1)$);
 \draw [color=blue, thick, ->] ($0.5*(V2)+0.5*(V3)$) --
 ($0.5*(V2)+0.5*(V3)+(a2)$);
 \draw [color=blue, thick, ->] ($0.5*(V3)+0.5*(V4)$) --
 ($0.5*(V3)+0.5*(V4)+(a3)$);
 \draw [color=blue,  thick, ->] ($0.5*(V4)+0.5*(V5)$) --
 ($0.5*(V4)+0.5*(V5)+(a4)$);
 \draw [color=blue, thick, ->] ($0.5*(V5)+0.5*(V1)$) --
 ($0.5*(V5)+0.5*(V1)+(a5)$);

\node[color=darkgreen, scale=1.4] at (null) {$\bullet$};

 \node[scale=1, left, yshift=0.3cm] at ($0.5*(V1)+0.5*(V2)$) {$ F_1$};
 \node[scale=1, above,xshift=0.3cm, yshift=0.1cm] at ($0.5*(V2)+0.5*(V3)$) {$ F_2$};
 \node[scale=1, right, xshift=0.0cm, yshift=0.0cm] at ($0.5*(V3)+0.5*(V4)$) {$ F_3$};
 \node[scale=1, below, xshift=0.0cm, yshift=-0.1cm] at ($0.5*(V4)+0.5*(V5)$) {$ F_4$};
 \node[scale=1, below, xshift=0.3cm, yshift=0.0cm] at
 ($0.5*(V5)+0.5*(V1)$) {$ F_5$};

 \node[scale=1, left] at  ($0.5*(V1)+0.5*(V2)+(a1)$) {$\va_1$};
 \node[scale=1, above] at ($0.5*(V2)+0.5*(V3)+(a2)$) {$\va_2$};
 \node[scale=1, right] at ($0.5*(V3)+0.5*(V4)+(a3)$) {$\va_3$};
 \node[scale=1, below] at ($0.5*(V4)+0.5*(V5)+(a4)$) {$\va_4$};
 \node[scale=1, below] at ($0.5*(V5)+0.5*(V1)+(a5)$) {$\va_5$};

\draw [fill opacity=0.1,fill=red] (V1) -- (V2) -- (null) --cycle;
\draw [fill opacity=0.1,fill=red] (V2) -- (V3) -- (null) --cycle;
\draw [fill opacity=0.1,fill=red] (V3) -- (V4) -- (null) --cycle;
\draw [fill opacity=0.1,fill=red] (V4) -- (V5) -- (null) --cycle;
\draw [fill opacity=0.1,fill=red] (V5) -- (V1) -- (null) --cycle;

\node[scale=1, right, yshift=0.2cm, color=def] at ($0.5*(V1)+0.5*(V2)$) {$C_1$};
 \node[scale=1, below,xshift=0.2cm, yshift=0.1cm, color=def] at ($0.5*(V2)+0.5*(V3)$) {$C_2$};
 \node[scale=1, left, xshift=0.0cm, yshift=-0.2cm, color=def] at ($0.5*(V3)+0.5*(V4)$) {$C_3$};
 \node[scale=1, above, xshift=-0.3cm, yshift=-0.1cm, color=def] at ($0.5*(V4)+0.5*(V5)$) {$C_4$};
 \node[scale=1, above, xshift=0.3cm, yshift=-0.1cm, color=def] at
 ($0.5*(V5)+0.5*(V1)$) {$C_5$};
\end{tikzpicture}
\caption{Cone-volumes of a polytope}
\end{figure}
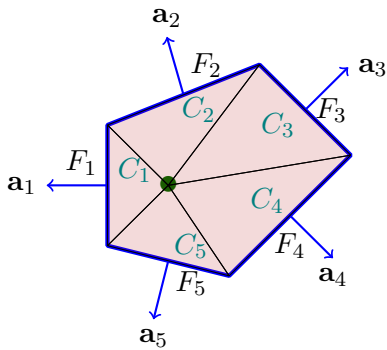

 The cone-volume measure of $P$ is given by (cf.~\eqref{eq:cvm})
\begin{equation*}
\V_P=\sum_{i=1}^m \V(C_i)\delta_{a_i},
\end{equation*}
where $\V(C_i)$ is the volume, i.e., $n$-dimensional Lebesgue measure,
of $C_i$ and $\delta_{a_i}$ denotes the delta measure concentrated on
$a_i$. Hence, $P$ satisfies  the subspace concentration condition
(cf.~\eqref{eq:scc}) if for  every subspace $L\subseteq\R^n$ 
\begin{equation}
  \sum_{a_i\in L}\V(C_i) \leq\frac{\dim L}{n}\,\V(P),
\label{eq:sccp}
\end{equation} 
and  equality holds in \eqref{eq:sccp} for a subspace $L$ if and only
if there exists a subspace $\overline{L}$, complementary to $L$, so
that $\{a_j: a_j\notin L\}\subset \overline{L}$. In other words, $A=(A \cap
L)\cup (A\cap\overline{L})$, where $A=\{a_1,\dots,a_m\}$.

In general, the cone-volume measure depends on the position of
the origin and not every $K\in\Kon$ fulfills  the subspace
concentration condition. In order to extend results from the
origin-symmetric case, in 
\cite[Problem 8.9]{Boroczky:2012vm} it is
asked whether the cone-volume measure of convex bodies having the centroid at the origin
satisfies the subspace concentration condition and our main result
gives an affirmative answer in the case of polytopes. 

\begin{theorem} Let $P\in\Kon$ be a polytope with centroid  at
  the origin. The cone-volume measure of $P$  satisfies the subspace
  concentration condition. 
\label{thm:main}
\end{theorem}
As mentioned before, in \cite{Boroczky:2012us} it was even shown that
the subspace concentration condition for even measures  characterizes
cone-volume measures of origin-symmetric convex bodies. Regarding
polytopes, the inequality \eqref{eq:sccp} was proven in the special
case $\dim L=1$ by Hern{\'a}ndez
Cifre (2007, private communication), in
the cases $\dim L=1,n-1$ by Xiong \cite{Xiong:2010fl}, and  for
origin-symmetric polytopes independently by Henk, Sch\"urmann \& Wills
\cite{Henk:2005vh} and by  He, Leng \& Li \cite{He:2006dm}.  

The motivation for studying the relation \eqref{eq:sccp} for the class
of polytopes  in \cite{He:2006dm, Xiong:2010fl} stems from  the
$\U$-functional of a polytope $P$ given by 
\begin{equation*} 
 \U(P)^n =\sum_{a_{i_1}\wedge\cdots\wedge a_{i_n}\ne 0}
  \V(C_{i_1})\cdot\ldots\cdot \V(C_{i_n}),
\end{equation*}  
where $a_{i_1}\wedge\cdots\wedge a_{i_n}\ne 0$ means that the vectors
are linearly independent.
This is a centro-affine functional, i.e., it is invariant with respect to
volume preserving linear transformations. It  was introduced by Lutwak,
Yang \& Zhang in \cite{Lutwak:2001dn} and has  proved useful in
obtaining strong 
inequalities for the volume of projection bodies of polytopes. For 
information on projection bodies we refer to the books by Gardner
\cite{Gardner:2006wn} and Schneider \cite{Schneider:1993td}, and for more
information on the importance of centro-affine functionals we refer to
 \cite{HaberlParapatits:2013, Ludwig:2010tm} and the references within.

 Here we are mainly
interested in the relation of $\U(P)$ to the volume of $P$. Obviously, $\U(P)\leq \V(P)$ and in \cite{Lutwak:2001dn} the problem was
posed that for polytopes with centroid at the origin $\U(P)$ is
bounded from below by 
\begin{equation}
    \U(P)\geq \frac{(n!)^{1/n}}{n}\V(P),
\label{eq:ufunctional} 
\end{equation} 
and equality holds if and only if $P$ is a parallelotope (see also \cite[Problem 8.6]{Boroczky:2012vm} for an extension to
convex bodies). Observe, for $n\to \infty$ the factor in the lower
bound becomes  $1/\mathrm{e}$ and so it is independent on the dimension. This is one
feature of the $\U$-functional making it so useful.    

In \cite{He:2006dm} it was shown by He, Leng \& Li that  \eqref{eq:ufunctional}  can be deduced from
\eqref{eq:sccp} and they proved \eqref{eq:ufunctional} (including the
  equality case)  for origin-symmetric
polytopes. Analogously, the results on  \eqref{eq:sccp}  in
\cite{Xiong:2010fl} were used in order to establish
\eqref{eq:ufunctional} (including the  equality case) 
for  arbitrary two- and three-dimensional polytopes  with
centroid at the origin. Following the lines of these results  we   prove \eqref{eq:ufunctional} in full generality. 
\begin{theorem} Let $P\in\Kon$ be a polytope with centroid  at
  the origin. Then 
\begin{equation*}
    \U(P)\geq \frac{(n!)^{1/n}}{n}\V(P),
\end{equation*} 
and equality holds if and only if $P$ is a parallelotope.
\label{thm:ufunctional}
\end{theorem}

Finally, we  remark that the logarithmic Minkowksi problem is a
particular case  of the $L_p$-Minkowski problem, one
of the central problems in  convex geometric analysis. 

\smallskip
\noindent{\bf $L_p$-Minkowski problem.} Find necessary and sufficient
conditions on a finite Borel measure $\mu$ on the unit sphere
$S^{n-1}$ so that $\mu$ is the $L_p$-surface area measure of a convex
body in $\R^n$. 

\smallskip
For details we refer to \cite{Boroczky:2012us, Lutwak:1993uh} and the references within. Here we just want
to mention that for $p=0$ the $L_0$-surface area measure is the
cone-volume measure, and B\"or\"oczky Jr., Lutwak, Yang
\& Zhang \cite{Boroczky:2012us} solved  the $L_0$-Minkowski
problem for even measures  via the subspace concentration condition.  
Theorem \ref{thm:main} shows that it is also a necessary condition  in the case of polytopes
with centroid at the origin. For results on the sufficiency of the
subspace concentration condition  in the planar polygonal case we refer to  \cite{Stancu:2002, Stancu:2003}.

The proof of Theorem \ref{thm:main}, which will be given in the 
last section, is  based on the Gaussian divergence theorem applied to
the log-concave function measuring the volume of slices of $P$ by
parallel planes. 

\section{Preliminaries}
In order to keep the paper largely self-contained, we collect here some
basic facts from Convex Geometry and Polytope Theory needed in our
investigations. Good general references on these topics are the books
by Barvinok \cite{Barvinok:2002vr}, Gardner \cite{Gardner:2006wn}, Gruber \cite{Gruber:2007um},
Schneider \cite{Schneider:1993td} and Ziegler \cite{Ziegler:1995gl}.

As usual, for two subsets  $C, D\subseteq \R^n$ and reals
$\nu,\mu \geq 0$ the  Minkowski combination is defined by 
\begin{equation*}
  \nu\,C+\mu\,D = \{\nu\,c+\mu\,d : c\in C,\,d\in D\}.
\end{equation*}  
By the celebrated Brunn-Minkowski inequality we know that the $n$-th root 
of the volume of the Minkowski combination  is a concave function. 
 More precisely, for two convex bodies
$K_0,K_1\subset\R^n$ and for $\lambda\in[0,1]$  we have 
\begin{equation}
   \V((1- \lambda)\,K_0+\lambda\, K_1)^{1/n}\geq
   (1-\lambda)\,V(K_0)^{1/n}+ \lambda\,\V(K_1)^{1/n}
\label{eq:brunn_minkowski} 
\end{equation}
with equality for some  $0<\lambda<1$ if and only if $K_0$ and $K_1$ lie
in parallel hyperplanes  or are homothetic, i.e., there exist a
$t\in\R^n$ and $\mu\geq 0$ such that $K_1=t+\mu\,K_0$ (see, e.g.,
\cite{Gardner:2002up}, \cite[Sect. 6.1]{Schneider:1993td}).

Let $f:C\to \R_{> 0}$ be a positive function on an open  convex subset
$C\subset\R^n$ with the property that
there exists a $k\in\N$ such that $f^{1/k}$ is concave. Then by the
(weighted) arithmetic-geometric mean inequality 
\begin{equation*}
\begin{split}
f((1-\lambda)\, x+\lambda\, y) & = 
 \left(f^{1/k}((1-\lambda)\, x+\lambda\, y)\right)^k \\ &\geq  \left((1-\lambda)
 f^{1/k}(x)+\lambda f^{1/k}(y)\right)^k\\ &\geq f^{1-\lambda}(x)\cdot f^{\lambda}(y).
\end{split}
\end{equation*}
This means that $f$ belongs to the class of log-concave functions
which by the positivity of $f$ is equivalent to 
\begin{equation*}
  \ln f((1-\lambda)\, x+\lambda\, y)\geq (1-\lambda)\ln f(x)+\lambda\ln f(y),
\end{equation*}   
for $\lambda\in [0,1]$. Hence,  for all $x,y\in C$ there exists a
subgradient $g(y)\in\R^n$ such that (cf., e.g., \cite[Sect. 23]{Rockafellar:1997ww}
\begin{equation}
\ln f(x)-\ln f(y)\leq\ip{g(y)}{x-y}.
\label{eq:gradient_concave}
\end{equation}
If $f$ is differentiable at $y$,  the subgradient  is the gradient
of $\ln f$ at $y$, i.e., $g(y)=\nabla \ln f = \frac{1}{f(y)}\nabla f(y)$.

 For a 
subspace $L\subseteq\R^n$, let  $L^\perp$ be its orthogonal
complement,  
and for $X\subseteq\R^n$ we denote by $X|L$ its orthogonal
projection onto $L$, i.e., the image of $X$ under the linear map
forgetting the part of $X$  belonging to $L^\perp$. 

Here, for a convex body $K\in\Kon$ and a $k$-dimensional subspace $L$,
$0<k<n$,  we are interested in the function measuring the volume of
$K$ intersected with planes parallel to $L^\perp$, i.e., in the function 
\begin{equation}
f_L: K|L \to \R_{\geq 0}\text{ with } x\mapsto \V_{n-k}(K\cap
(x+L^\perp)), 
\label{eq:function}
\end{equation} 
where $\V_{n-k}(\cdot)$ denotes the $(n-k)$-dimensional volume.  By
the Brunn-Minkowski inequality and the remark above, $f_L$ is a log-concave
function which is positive at least in   the (relative) interior of
$K|L$ (cf.~\cite{Ball:1988vo}). $f_L$ is also called the $(n-k)$-dimensional  X-ray of $K$
parallel to $L^\perp$ (cf.~\cite[Chapter 2]{Gardner:2006wn}).

Next we want to consider this function for a polytope
$P\in\Kon$. Such a polytope may be represented as 
\begin{equation*}
  P=\{x\in\R^n : \ip{a_i}{x}\leq b_i,\, 1\leq i\leq m\},
\label{eq:polytope}
\end{equation*}
where $a_i\in S^{n-1}$ are the outer unit normals of the facets of $P$, i.e., all
$(n-1)$-dimensional faces of $P$ are given by $F_i=P\cap \{x\in\R^n :
\ip{a_i}{x}= b_i\}$, $1\leq i\leq m$. Since $0\in\inte P$ we have
$b_i>0$, and since $a_i\in S^{n-1}$, $b_i$ is the distance of $F_i$
from the origin. Thus 
\begin{equation}
  \V(P)=\frac{1}{n}\sum_{i=1}^m V_{n-1}(F_i)\,b_i = \sum_{i=1}^m\V(C_i),
\label{eq:polytope_volume}
\end{equation} 
where $C_i=\conv\{0,F_i\}$.  The boundary
$\bd P$ of $P$ is the union of the facets of $P$. In general, for a fixed
$k\in\{0,\dots,n-1\}$, the union of all
$k$-faces of $P$  is called the $k$-skeleton of $P$. The
orthogonal projection of the $(k-1)$-skeleton onto a $k$-dimensional
plane $L$ induces a polytopal subdivision $\mathcal{D}(P)_{L}$ of
$P|L$, i.e., $\mathcal{D}(P)_{L}$ is a collection of $k$-dimensional
polytopes having pairwise disjoint interiors, the intersection of any two
of them is a face of both, the union covers $P|L$ and the preimage of
the boundary of a polytope in $\mathcal{D}(P)_{L}$ is contained in a
$(k-1)$-face of $P$.   

It was used, observed and proved in
different contexts that $f_L:P|L\to \R_{\geq 0}$ is a piecewise
polynomial function, actually a spline for a generic subspace $L$. 
Here we will only  use the following
result as stated by  Gardner and Gritzmann {\cite[Proposition 3.1]{Gardner:1994ub}}.

\begin{proposition} Let $L$ be a $k$-dimensional subspace,
  $0<k<n$. The function $f_L:P|L\to \R_{\geq 0}$  is a piecewise
polynomial function of degree at most $n-k$; more precisely, on every
$k$-dimensional polytope of the subdivision $\mathcal{D}(P)_{L}$ it is
a polynomial of degree at most $n-k$.
\label{prop:subdivision}
\end{proposition}

The centroid $c(S)$ of a set  $S\subset\R^n$ with $\V(S)>0$  
 is defined as 
\begin{equation*}
   c(S)=\frac{1}{\V(S)}\int_S x\, \mathrm{d}x, 
\end{equation*}
where  $\mathrm{d}x$ is the abbreviation for
$\mathrm{d}\mathcal{H}^n(x)$. For lower dimensional sets
$S\subset\R^n$, $S\ne\emptyset$,  the
centroid $c(S)$ is calculated with respect to the space given by the
affine hull of $S$. 

For $K\in\Kon$ with $c(K)=0$ and a subspace $L$ with $0<\dim L<n$
we have by Fubini's theorem with respect to the decomposition $L\oplus L^{\perp}$   
\begin{equation*}
\begin{split}
0 & =\int_K x\, \mathrm{d}x \\ &= \int_{K|L}\left(
  \int_{(\hat{x}+L^\perp)\cap K} \tilde{x}\,
  \mathrm{d}\tilde{x}\right) \mathrm{d}\hat{x}\\ & = \int_{K|L}
   f_L(\hat{x})\, c((\hat{x}+L^\perp)\cap K)\,\mathrm{d}\hat{x}.
\end{split}
\end{equation*}
Writing $c((\hat{x}+L^\perp)\cap K)=\hat{x}+\tilde{y}$ with
$\tilde{y}\in L^\perp$ gives 
\begin{equation}
 \int_{K|L}  f_L(\hat{x})\,\hat{x}\,\mathrm{d}\hat{x} =0,
\label{eq:centroid}
\end{equation}
i.e., the first moment of $f_L$  vanishes. This will be the main 
property of the centroid used later on. Indeed, we will need it in
order to apply the following lemma on log-concave functions. 

\begin{lemma} Let $C\in\Kon$, and let  $f:\inte C\to\R_{>0}$ be a log-concave
  function with $\int_C f(x)\,x\,\dbm x=0$. 
 Furthermore, assume that $\nabla f(x)$ exists  almost everywhere on
 $\inte C$, 
   and that also $\int_C \langle x, \nabla f(x)\rangle \,\dbm x$ exists. Then  
\begin{equation*} 
\int_{C} \langle x, \nabla
  f(x)\rangle\, \dbm x\leq 0, 
\end{equation*}
 with equality if and only if there exist
  $c\in\R^n$, $\gamma\in\R_{>0}$ with
  $f(x)=\gamma\,\mathrm{e}^{\langle c,x\rangle}$.
\label{lem:logconcave}
\end{lemma}
\begin{proof} By the concavity of $\ln f(x)$  we have for all $x,y\in
  \inte C$ (cf.~\eqref{eq:gradient_concave})  
\begin{equation}
\ln f(x)-\ln f(y) \leq 
\ip{g(y)}{x-y},
\label{eq:derivative}
\end{equation}
where $g(y)$ is a subgradient at $y$.
Interchanging the role of $x$ and $y$ and adding  leads to 
\begin{equation*}
0\leq \ip{g(y)-g(x)}{x-y}. 
\end{equation*}
Setting $y=0$ leads to $\ip{g(x)}{x}\leq \ip{g(0)}{x}$. For points
$x\in C^\prime$, where $C^\prime\subseteq C$ is the set
where  $\nabla f(x)$ exists,  this gives   
\begin{equation*}
\ip{\nabla f(x)}{x} \leq \ip{g(0)}
{f(x)\, x}. 
\end{equation*}
Hence in view of our assumption on $\nabla f$ and on the first moment of $f$ on $C$  we get  
\begin{equation*}
\begin{split}
\int_C \ip{x}{\nabla f(x)} \dbm x   & = \int_{C^\prime} \ip{x}{\nabla f(x)} \dbm x\\
&\leq 
\int_{C^\prime} \ip{g(0)} {f(x)\, x} \dbm x\\  
&=\ip{g(0)}{\int_C  
f(x)\, x\, \dbm x} =0. 
\end{split}
\end{equation*}
If the inequality holds with equality,  we must have almost everywhere equality in
\eqref{eq:derivative} for $y=0$. Hence, $\ln f(x)$ is an affine
function. Together with the positivity of $f$ on $\inte C$ there exist
  $c\in\R^n$, $\gamma\in\R_{>0}$ with
  $f(x)=\gamma\,\mathrm{e}^{\ip{c}{x}}$. On the other hand,
  if $f$ is of this form then $\nabla f(x)=f(x)\,c$ and so 
\begin{equation*} 
\int_C \ip{x}{\nabla f(x)} \dbm x =\ip{c}{\int_C
   f(x) x\, \dbm x} =0.
\end{equation*}

\end{proof}

\section{Constant volume Sections}
In order to treat the equality case in the subspace concentration
condition of Theorem \ref{thm:main} we need the following characterization. 
\begin{lemma} Let $P\in\Kon$ be a polytope, let $A$ be the set of its
  outer unit normals, and let $L\subset\R^n$ be a $k$-dimensional
  subspace, $0<k<n$. 
Then $f_L:P|L\to\R_{\geq 0}$  is a
  constant function if and only if there exists a subspace
  $\overline{L}$, complementary to $L$, such that 
\begin{equation*}
    A=(A\cap L)\cup (A\cap \overline{L}).
\end{equation*}
\label{lem:constant_section}
\end{lemma} 
\begin{proof} Suppose $f_L(x)=f_L(0)$ for all $x\in   P|L$.  Then, in
  particular, 
\begin{equation*}
 f_L((1-\lambda)x +\lambda\, 0)^{1/(n-k)} = (1-\lambda)f_L(x)^{1/(n-k)}+  \lambda\,f_L(0)^{1/(n-k)}
\end{equation*}
for all $\lambda\in[0,1]$ and $x\in P|L$. Hence we have equality 
in the Brunn-Minkowski inequality \eqref{eq:brunn_minkowski}
  applied in the space $L^\perp$ and thus, 
for every 
$x\in P|L$ there exists a uniquely determined  $t(x)\in\R^n$ such that 
\begin{equation*}
     (x+L^\perp)\cap P = t(x)+(L^\perp\cap P).
\end{equation*}  
Let   $t:P|L\to\R^n$ be the associated map sending  $x\mapsto
t(x)$. Then $t(\cdot)$ is injective and convex linear,
i.e., $t((1-\lambda)x+\lambda y)=(1-\lambda)t(x)+\lambda t(y)$ for
$x,y\in P|L$ and $\lambda\in [0,1]$. Thus it is an affine function,
and since 
$t(0)=0$ we conclude that $t$ is linear. Hence,  $\widetilde{L}=\lin
t(P|L)$, i.e., the linear hull  
of $t(P|L)$,  is a $k$-dimensional linear subspace and we have 
\begin{equation*}
      P=(P\cap \widetilde{L})+ (P\cap L^\perp).
\label{eq:poly_decom}
\end{equation*}
Since $P\cap \widetilde{L}$ is a $k$-dimensional polytope and $(P\cap
L^\perp)$ an $(n-k)$-dimensional polytope,  the  facets of $P$ are 
given by 
\begin{equation*}
           \tilde{F}+ (P\cap L^\perp)\text{ or } F+(P\cap \widetilde{L}),
\end{equation*}
where $\tilde{F}$ is a facet, i.e., a $(k-1)$-face of  $P\cap
\widetilde{L}$ and $F$ is a facet, i.e., a $(n-k-1)$-face of  $P\cap
L^\perp$.
In the first case the outer unit normal  of such a facet is contained in
$(L^\perp)^\perp=L$, and in the latter case in
$\widetilde{L}^\perp$. Hence $A=(A\cap L)\cup (A\cap
\widetilde{L}^\perp)$, and  
since $P$ is bounded we also know that
$\widetilde{L}^\perp$ is complementary to $L$.

On the other hand, if  we have $A=(A\cap L)\cup (A\cap \overline{L})$
for complementary subspaces $L,\overline{L}$, then it is easy to see
that 
\begin{equation*}
      P=(P\cap L^\perp) + (P\cap \overline{L}^\perp).
\end{equation*}
In particular, by the complementarity of the subspaces we know that
for every  $x\in P|L$ there exists an unique
$\overline{t}(x)\in P\cap \overline{L}^\perp$ with $\overline{t}(x)|L=x$.
Hence, $P\cap (x+L^\perp)=  \overline{t}(x)+(P\cap L^\perp)$ for every $x\in
P|L$, which shows $f_L(x)=f_L(0)$.
\end{proof}

Lemma \ref{lem:constant_section}  also allows us to give a weak
generalization  of a characterization of
parallelotopes due to Guggenheimer \& Lutwak
\cite{Guggenheimer:1976wa}. It  will be used for the discussion
of the equality case in Theorem \ref{thm:ufunctional}.

 Here we need the following notation: for $n$ linearly independent
unit vectors $V=\{v_1,\dots,v_n\}$ and a $k$-subset $I\subset\{1,\dots,n\}$,
$0<k<n$,  we
denote by $L_I(V)=\lin\{v_j :j \in I\}$ the $k$-dimensional subspace
generated by this selection of vectors.

\begin{lemma} Let $P\in \Kon$ be a polytope and let  $0<k<n$. There
  exist $n$ linearly independent unit vectors $V=\{v_1,\dots,v_n\}$ such
  that the function  $f_{L_I(V)}:P|L_I(V)\to \R_{\geq
    0}$ is constant for every $k$-subset $I\subset\{1,\dots,n\}$ if and only if $P$
  is a parallelotope. 
\label{lem:parallelotope}
\end{lemma}
Before giving the proof we want to remark that for arbitrary convex
bodies and $k=1$ the result was shown by  Guggenheimer \& Lutwak
\cite{Guggenheimer:1976wa}. In fact, they only assumed that the
function is constant for $(n-1)$-many $1$-subsets. 

\begin{proof} The outer unit normals $\pm v_i$, $1\leq i\leq
  n$,  of a parallelotope  are always contained in complementary
  subspaces. Hence the sufficiency follows from Lemma
  \ref{lem:constant_section}.

Now let  vectors $V=\{v_1,\dots,v_n\}$ be given such that
$f_{L_I(V)}:P|L_I(V)\to\R_{\geq 0}$ is a constant function for any $k$-subset
$I\subset\{1,\dots,n\}$. For short we will  write $L_I$ instead of
$L_I(V)$.  According to Lemma \ref{lem:constant_section}
there exists for any $k$-subset $I$ a complementary subspace $\overline{L_I}$ such that 
\begin{equation}
        A=(A\cap L_I)\cup(A\cap\overline{L_I}),      
\label{eq:split}
\end{equation}
where $A$ is the set of outer unit normals of the polytope
$P$. Since $\dim A=n$ we have 
\begin{equation}
         \dim(A\cap L_I)=k\text{ and }\dim(A\cap\overline{L_I})=n-k,
\label{eq:dim}
\end{equation}
where, in general, $\dim X$ is the dimension of the affine hull of
$X\subset \R^n$. 

For a subset $I$ let $I_c$ be its complement
with respect to $\{1,\dots,n\}$, i.e., $I_c=\{1,\dots,n\}\setminus I$. We claim 
\begin{equation}
        \overline{L_I} = L_{I_c}.    
\label{eq:claim}
\end{equation}
Assume first $k\leq n/2$, and let $j\in I_c$. Then we may complete 
$\{j\}$ to a $k$-subset $J$ with $I\cap J=\emptyset$ and thus $L_I\cap
L_J=\{0\}$. By
\eqref{eq:dim}, applied to $J$, there exist $a_{j_1},\dots,a_{j_k}\in A\cap L_J$ with
$v_j\in\lin\{a_{j_1},\dots,a_{j_k}\}$. On the other hand, since $L_I\cap
L_J=\{0\}$ we know  by \eqref{eq:split} that
$a_{j_1},\dots,a_{j_k}\in A\cap\overline{L_I}$.  Thus $v_j\in
\overline{L_I}$ and so $L_{I_c}\subset \overline{L_I}$. Since both
subspaces are of dimension $n-k$ we are done. 

Now let $k>n/2$ and assume that there exists an $a^*\in A\cap
\overline{L_I}$ such that $a^*\notin L_{I_c}$. Without loss of
generality  let $I=\{1,\dots,k\}$, and let 
\begin{equation*}
a^*=\sum_{i=1}^k \alpha_i\,v_i+\sum_{j={k+1}}^n\beta_j\,v_j
\end{equation*}
with some $\alpha_i,\beta_j\in\R$. 
Since $a^*\notin L_{I_c}\cup L_I$  we may assume
$\alpha_1\cdot\beta_n\ne 0$.  Then, for $J=\{n-k+1,\dots,n\}$  we get by
 \eqref{eq:split}  that $a^*\in A\cap \overline{L_J}$ and thus
\begin{equation}
        a^*\in (A\cap \overline{L_I})\cap(A\cap \overline{L_J}).
\label{eq:section}
\end{equation}
On the other hand, \eqref{eq:split} also implies that  each $a\in A\setminus((A\cap
\overline{L_I})\cup (A\cap \overline{L_J}))$ belongs to
$\lin\{v_{n-k+1},\dots,v_k\}$ which finally gives 
\begin{equation*}
   A=(A\cap\lin\{v_{n-k+1},\dots,v_k\})\cup (A\cap \overline{L_I})\cup (A\cap \overline{L_J})
\end{equation*}
By \eqref{eq:section} and \eqref{eq:dim} the union of the last two
sets is contained in a linear subspace of dimension $2(n-k)-1$ which
yields the contradiction that $\dim A\leq n-1$. Hence, $(A\cap
\overline{L_I})\subset L_{I_c}$ and on account of \eqref{eq:dim} we get 
 \eqref{eq:claim}.

Now let $a\in A$ and  let $a=\sum_{i=1}^n\alpha_i\,v_i$ for some
scalars $\alpha_i\in\R$. Suppose two of them are non zero, and let
$j_1,j_2$ be the corresponding indices. Let $J$ be a $k$-subset
containing $j_1$ but not $j_2$. Then $a\notin (A\cap L_{J})\cup (A\cap
L_{J_c})$ contradicting \eqref{eq:claim} and \eqref{eq:split}.  
 Hence, $A\subset\{\pm v_i: 1\leq i\leq n\}$, and  since none strict subset of the
latter set can be the outer unit normals of a bounded set we
conclude 
\begin{equation*}
A=\{\pm v_i: 1\leq i\leq n\},
\end{equation*} 
and $P$ is a parallelotope.
\end{proof}

\section{Proof of theorems}
The proof of Theorem \ref{thm:main} relies on the Gaussian divergence
theorem, which is usually stated in the form  (cf., e.g., \cite{Konig:1964vd})
\begin{equation*}
\int_{V} \dive F(x)\, \mathrm{d}{x}=\int_{\bd' V}\ip{F(x)}{\nu(x)}\, \mathrm{d}{\mathcal{H}^{n-1}(x)},
\end{equation*}
where $V\subset\R^n$ is a compact subset with a piecewise smooth boundary, $F:\R^n\to\R^n$ is a
continuously differentiable vector field in an open neighborhood of
$V$, $\bd' V$ is that part of the boundary of $V$ admitting a unique
outer normal  $\nu(x)$ in $x\in\bd V$, and
$\dive F$ is the divergence of the vector field $F$, i.e., $\dive F
=\sum_{i=1}^n \frac{\partial F_i}{\partial x_i}$ where
$F(x)=(F_1(x),\dots, F_n(x))^\intercal$. 

Here we want to apply this theorem to the vector field 
\begin{equation}
 F_L:P|L\to L \text{ with } x\mapsto f_L(x)\,x,
\label{eq:our_funct}
\end{equation}
where $P\in\Kon$ is a polytope, $L$ is a $k$-dimensional linear
subspace, $0<k<n$, and  $f_L:P|L\to\R_{\geq
0}$ is the volume intersection function
$f_L(x)=\V_{n-k}((x+L^\perp)\cap P)$. This vector field is, in general, not
continuously differentiable. There are, however, numerous extensions of the divergence theorem
to much more general sets  than  compact sets and to functions with
certain singularities which also cover the case we need (see, e.g., \cite{Pfeffer:1990ts}). On the other hand, our vector field \eqref{eq:our_funct} is
``just'' a piecewise polynomial vector field, and so we briefly state how the Gaussian divergence theorem can be applied in 
our setting.   
 
\begin{figure}[tb]
\begin{tikzpicture}[scale=0.7, line join=bevel, x={({cos(0)*1cm},{sin(0)*1cm})}, y={(2.85mm, 1.6mm)}, z={(-0.0cm,1.cm)}]

\coordinate (B1) at (1,-1,3);
\coordinate (B2) at (1,1,3);
\coordinate (B3) at (-3,1,3);
\coordinate (B4) at (-3,-1,3);

\coordinate (C1) at (1,-3,1);
\coordinate (C2) at (1,3,1);
\coordinate (C3) at (-1,3,1);
\coordinate (C4) at (-1,-3,1);

\draw [fill opacity=0.1,fill=gray] (B2) -- (B3) -- (C3) -- (C2)--cycle;
\draw [fill opacity=0.1,fill=gray] (C1) -- (C2) -- (C3) -- (C4) --
cycle;
  \draw [fill opacity=0.1,fill=gray] (B1) -- (C1) -- (C2) -- (B2) --
  cycle;
  \draw [fill opacity=0.1,fill=gray] (B3) -- (C3) -- (C4) -- (B4) --
  cycle;
 \draw [fill opacity=0.1,fill=gray] (B4) -- (B1) -- (C1) -- (C4)--cycle;
\draw [fill opacity=0.1,fill=gray] (B1) -- (B2) -- (B3) -- (B4) --
cycle;

\draw [color=blue, ultra thick] (C1) -- (C2) -- (C3) -- (C4) -- cycle;
\draw [color=blue, ultra thick] (B1) -- (C1) -- (C2) -- (B2) -- cycle;
\draw [color=blue, ultra thick] (B3) -- (C3) -- (C4) -- (B4) --cycle;
\draw [color=blue, ultra thick] (B4) -- (B1) -- (C1) -- (C4)--cycle;
\draw [color=blue, ultra thick] (B2) -- (B3) -- (C3) -- (C2)--cycle;
\draw [color=blue, ultra thick] (B1) -- (B2) -- (B3) -- (B4) --cycle;

\node[color=red, scale=1.2] at (B1) {$\bullet$};
\node[color=red, scale=1.2] at (B2) {$\bullet$};
\node[color=red, scale=1.2] at (B3) {$\bullet$};
\node[color=red, scale=1.2] at (B4) {$\bullet$};
\node[color=red, scale=1.2] at (C1) {$\bullet$};
\node[color=red, scale=1.2] at (C2) {$\bullet$};
\node[color=red, scale=1.2] at (C3) {$\bullet$};
\node[color=red, scale=1.2] at (C4) {$\bullet$};

\draw [fill opacity=0.1, fill=green] (-4,-5,-2) -- (-4,5,-2) --
(4,5,-2) -- (4,-5,-2)  --cycle;

\draw [fill opacity=0.2, fill=gray]  (1,3,-2) --
(-1,3,-2) -- (-3,1,-2) -- (-3,-1,-2) -- (-1,-3,-2) -- (1,-3,-2) --cycle;
 \draw [color=blue, ultra thick] (1,3,-2) --
 (-1,3,-2) -- (-3,1,-2) -- (-3,-1,-2) -- (-1,-3,-2) -- (1,-3,-2) --cycle;

\node[color=red, scale=1.] at (1,3,-2) {$\bullet$};
\node[color=red, scale=1.] at (-1,3,-2) {$\bullet$};
\node[color=red, scale=1.] at (-3,1,-2) {$\bullet$};
\node[color=red, scale=1.] at (-3,-1,-2) {$\bullet$};
\node[color=red, scale=1.] at (-1,-3,-2) {$\bullet$};
\node[color=red, scale=1.] at (1,-3,-2) {$\bullet$};

\draw [color=blue, densely dashed, ultra thick] (1,-1,-2) -- (1,1,-2)
-- (-3,1,-2) -- (-3,-1,-2)  --cycle;

\draw [color=blue, densely dashed, ultra thick]   (1,3,-2) --
(-1,3,-2) -- (-1,-3,-2) -- (1,-3,-2) --cycle;

\node[scale=1.0] at (4,4,-1.8) {$L$};
\node[scale=1.1] at (-3,3.4,-1.96) {$P|L$};

\draw[color=orange, ultra thick, ->] (-3.2,0,1) -- (-3.2,0,-0.5); 
\node[scale=1.0] at (-3.3,-0.7,0.25) {$L^\perp$};
\node[scale=1.1] at (1,3,1.6) {$P$};
\node[scale=1] at (0,0.1,-1.9) {$P_i$};
\end{tikzpicture}
\caption{The polytopal subdivision induced by the orthogonal projection of the
$(k-1)$-skeleton of $P$ onto $L$}
\end{figure}
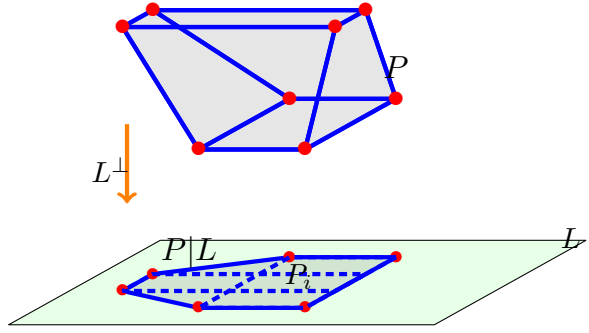

\begin{lemma}\label{gausslemma}
Let $P$ be a polytope, and let $L\subset\R^n$ be
a $k$-dimensional linear subspace, $0<k<n$.  Let $\mathcal{D}(P)_{L}=\{P_1,\dots,P_r\}$
be the polytopal subdivision induced by the orthogonal projection of the
$(k-1)$-skeleton $S$ of $P$ onto $L$.  
Let $F:P|L\to L$ be a
vector field which is a polynomial in each component and  on each
 polytope $P_i\in \mathcal{D}(P)_{L}$. Then 
\begin{equation*}
\int_{(P|L)\setminus(S|L)}\dive F(x)\,\dbm{x}=
\int_{\bd'(P|L)} \ip{F(x)}{a(x)}\,\dbm{\mathcal H}^{k-1}(x).
\end{equation*}
Here $a(x)\in L$ is the unique outer normal of the polytope $P|L$ in the
boundary point $x\in \bd'(P|L)$.
\end{lemma}
\begin{proof}  Since $F$ is a polynomial vector field on 
each $P_i$, $F$ can canonically be extended to
an open neighborhood of $P_i$, and hence we can use the  divergence theorem of Gauss and get
\begin{equation*}
\int_{P_i}\dive F(x)\,\dbm{x}=\int_{\bd'(P_i)} \langle F(x),a_i(x)\rangle\,\dbm{\mathcal H}^{k-1}(x),
\end{equation*}
with $a_i(x)\in L$ being the unique outer normal of  $P_i$ in $x\in
\bd'(P_i)$. 
Thus, in particular,  $\int_{P_i \setminus(S|L)}\dive F(x)\,\dbm{x}$
is well defined and  since $S|L$ is a set of measure $0$, we have
\begin{equation*}
\int_{(P|L)\setminus(S|L)} \dive F(x)\,\dbm{x}=
\sum_{i=1}^r \int_{\bd'(P_i)}\ip{F(x)}{a_i(x)}\,\dbm{\mathcal H}^{k-1}(x).
\end{equation*}
Every $x\in\bd'(P_i)\setminus\bd(P|L)$ is contained in exactly one
more $P_j$, $j\neq i$, 
and $a_j(x)=-a_i(x)$. Hence,   with
$a_i(x)=a(x)$ for $x\in\bd'(P_i)\cap\bd'(P|L)$ we get 
\begin{equation*}
\sum_{i=1}^r  \int_{\bd'(P_i)}\ip{F(x)}{a_i(x)}\,\dbm{\mathcal H}^{k-1}(x)=
\int_{\bd'(P|L)}\,\ip{F(x)}{a(x)}\dbm{\mathcal H}^{k-1}(x),
\end{equation*}
which finishes the proof.
\end{proof}

We are now ready to prove Theorem \ref{thm:main}. 
\begin{proof}[Proof of Theorem \ref{thm:main}] Let $P\in\Kon$ be a
  polytope with centroid at the origin. Let $F_1,\dots,F_m$ be the
  facets of $P$, and let
$a_i\in S^{n-1}$ be the outer unit normal of the facet $F_i$,
$1\leq i\leq m$.  Let $A= \{a_1,\dots,a_m\}$, and let $b_i>0$ be the
distance of the facet $F_i$ from the origin.  Let $L$ be a
 $k$-dimensional subspace with $0<k< n$. We have to show
 (cf.~\eqref{eq:sccp} and \eqref{eq:polytope_volume})
\begin{equation}
  \sum_{a_i\in L}\V_{n-1}(F_i)\,b_i  \leq k\,\V(P) 
  .
\label{eq:toshow}
\end{equation}
with equality if and only if there exists
a subspace $\overline{L}$, complementary to $L$, so
that $A=(A \cap
L)\cup (A\cap\overline{L})$. 

According to Proposition \ref{prop:subdivision}, the vector field $F_L(x)=f_L(x)\,x$
(cf.~\eqref{eq:our_funct}) satisfies the assumptions of Lemma
\ref{gausslemma} and on account of 
\begin{equation*}
\dive F_L(x)= k\,f_L(x) + \ip{\nabla f_L(x)}{x}  
\end{equation*}
we get 
\begin{equation}
\begin{split}
\int_{\bd'(P|L)}  &f_L(x)\ip{x}{a(x)}\dbm\mathcal{H}^{k-1}(x)  \\
&=k\,\int_{(P|L)\setminus(S|L)}\,f_L(x)  \dbm{x} + \int_{(P|L)\setminus(S|L)}\,\ip{\nabla f_L(x)}{x}\dbm{x}
\\ 
&=k\,\V(P)+\int_{(P|L)\setminus(S|L)}\,\ip{\nabla f_L(x)}{x}\dbm{x},\\
\end{split}
\label{eq:together}
\end{equation}
where in the last step we used again  that $S|L$ is a
 set of measure $0$.

Now let $\tilde{a}_1,\dots,\tilde{a}_l$ be the outer unit normals of
the facets of $P|L$, i.e.,
the $(k-1)$-faces of $P|L$, having distance $\tilde{b}_i$ to the origin.  Let 
$\tilde{F}_i=P\cap\{x\in\R^n : \langle \tilde{a}_i,x\rangle
=\tilde{b}_i\}$, $1\leq i\leq l$, be the faces of $P$ projected onto
the facets of $P|L$. Taking into account that $f_L$  measures the
$(n-k)$-dimensional volume we have 
\begin{equation*}
  \int_{\bd'(P|L)} f_L(x) \ip{x}{a(x)} \,\dbm \mathcal{H}^{k-1}(x) = \sum_{i=1}^l
 \V_{n-1}(\tilde{F}_i)\,\tilde{b}_i. 
\end{equation*}
Hence $\tilde{F}_i$ contributes to the above sum only when it is a  facet
of $P$, i.e., $F_j=\tilde{F}_i$, $a_j=\tilde{a}_i\in L$ and
$b_j=\tilde{b}_i$ for a certain $j\in\{1,\dots,m\}$.  Thus we may write (cf.~\eqref{eq:together})
\begin{equation}
\sum_{a_i\in L}\V_{n-1}(F_i)\,b_i = k\,\V(P) + 
\int_{(P|L)\setminus(S|L)}\,\ip{\nabla f_L(x)}{x}\dbm{x}. 
\label{eq:final}
\end{equation}
Since $0$ is the centroid of $P$ we have (cf.~\eqref{eq:centroid}) 
\begin{equation*}
  \int_{P|L} f_L(x)\, x \,\dbm x=0, 
\end{equation*}
and since $S|L$ is a set of measure $0$ we may apply  Lemma
\ref{lem:logconcave} to $f_L$. Thus 
\begin{equation}  
\int_{(P|L)\setminus(S|L)} \ip{\nabla
f_L(x)}{x} \,\dbm x\leq 0,
\label{eq:int_gradient}
\end{equation}
which yields \eqref{eq:toshow}  by \eqref{eq:final}. 

Now suppose we have equality in \eqref{eq:toshow}. Then we also have
equality in \eqref{eq:int_gradient} and by Lemma
\ref{lem:logconcave} there exist $\gamma >0$, $c\in\R^n$ such that 
$f_L(x)=\gamma\,\mathrm{e}^{\ip{c}{x}}$. Since the $(n-k)$-th root of
$f_L(x)$ is concave we must have $c=0$, i.e., $f_L(x)$ is a constant
function. Thus, by Lemma \ref{lem:constant_section}, there exists a
complementary subspace $\overline{L}$ with $A=(A\cap
L)\cup(A\cap\overline{L})$. 

On the other hand, if we have such a partition of $A$ into
complementary subspaces $L$ and $\overline{L}$, $\dim L=k$ and
$\dim\overline{L}=n-k$, then we may either apply Lemma
\ref{lem:constant_section} and then \eqref{eq:final}, or we just observe
that in this case we may write (cf.~\eqref{eq:polytope_volume})
\begin{equation*}
\begin{split} 
k\,\V(P) +(n-k)\,\V(P) &= n \V(P)\\ 
       &=\sum_{a_i\in A\cap L} \V_{n-1}(F_i)\,b_i + \sum_{a_i\in A\cap \overline{L}} \V_{n-1}(F_i)\,b_i.
\end{split}      
\end{equation*}
Hence, in view of the validity of the  inequality  \eqref{eq:toshow} for $L$ and
$\overline{L}$,  we  have actually equality  in \eqref{eq:toshow} for
$L$ and $\overline{L}$.
\end{proof}

Next we come to the proof of Theorem \ref{thm:ufunctional}, and here 
 we follow  the approach of He, Leng \& Li \cite{He:2006dm}.
\begin{proof}[Proof of Theorem \ref{thm:ufunctional}] Let $P\in\Kon$
  be a polytope with centroid at the  origin. Let $F_1,\dots,F_m$ be the facets
  of $P$ with associated outer unit normals  $a_1,\dots,a_m$ and
  let $C_i=\conv\{0, F_i\}$, $1\leq i\leq m$. For $1\leq k\leq n$ we set 
\begin{equation}
   \sigma_k(P)^k=\sum_{a_{i_1}\wedge \cdots\wedge a_{i_k}\ne 0}
  \V(C_{i_1})\cdot\ldots\cdot \V(C_{i_k}).
\label{eq:generalu}
\end{equation}
We have to show $\U(P)^n=\sigma_n(P)^n\geq n!/n^n\,\V(P)^n$ with equality
if and only if $P$ is a parallelotope.  On account of Theorem
\ref{thm:main} (cf.~\eqref{eq:sccp}) we may write  for $0<k<n$
\begin{equation}
\begin{split}
\sigma_{k+1}(P)^{k+1} & = \sum_{a_{i_1}\wedge \cdots \wedge a_{i_{k}}\ne 0}
\V(C_{i_1})\cdot\ldots\cdot \V(C_{i_{k}}) \times\\ 
 &\quad\quad\quad\hspace{2cm}
 \left(\V(P)-\sum_{a_l\in\lin\{a_{i_1},\cdots,a_{i_{k}}\}} 
     \V(C_l)\right) \\
&\geq  \sum_{a_{i_1}\wedge\cdots\wedge a_{i_{k}}\ne 0}
\V(C_{i_1})\cdot\ldots\cdot
\V(C_{i_{k}})\left(\V(P)\left(1-\frac{k}{n}\right)\right)\\ 
&=\frac{n-k}{n}\V(P)\,\sigma_{k}(P)^{k}.
\end{split}
\label{eq:recursion}
\end{equation}
Since $\sigma_1(P)=\V(P)$  this recursion gives  
\begin{equation*}
 \U(P)^n=\sigma_n(P)^n\geq
 \frac{1}{n}\V(P)\,\sigma_{n-1}(P)^{n-1}\geq \cdots\geq \frac{(n-1)!}{n^{n-1}}\V(P)^n,
\end{equation*}
which is the desired inequality.

 Having equality we must have equality  in each step of  
the recursion \eqref{eq:recursion}. Hence for any $k$-subset $I\subset\{1,\dots,m\}$
such that $L_{I}=\lin\{a_j : j\in I\}$ is of dimension $k$, we have 
\begin{equation*}
\sum_{a_l\in L_{I}}   \V(C_l) = \frac{k}{n}\V(P).
\end{equation*}
Thus by the equality case of Theorem \ref{thm:main}
(cf.~\eqref{eq:sccp}) we can find a complementary subspace
$\overline{L}_{I}$ with $A=(A\cap L_{I})\cup
(A\cap \overline{L}_{I})$, where $A$ is the set of all
outer unit normals of $P$. By Lemma \ref{lem:constant_section} this shows that
$f_{L_{I}}$ is a constant function for any  $k$-subset $I\subset\{1,\dots,m\}$
such the vectors $a_j$, $j\in I$, are linearly independent. Since there
are $n$ linearly independent vectors in $A$ we are in the position to
use Lemma \ref{lem:parallelotope} which gives that $P$ is a
parallelotope. 

On the other hand if $P$ is a parallelotope then
$\V(C_i)=\frac{1}{2\,n}\V(P)$ for all cones and thus we have equality.   
\end{proof}

\bigskip
\noindent 
{\it  Acknowledgement.} The authors thank  Rolf Schneider,
Eugenia Saor{\'\i}n G{\'o}mez, Guangxian Zhu and the referees for their very helpful
comments and suggestions.


\end{document}